\documentclass[10pt,reqno]{amsart}

\usepackage{amsfonts,amsthm,latexsym,amsmath,amssymb,amscd,amsmath, mathrsfs, epsf}
\usepackage{graphicx}
\newtheorem{theorem}{Theorem}
\newtheorem{proposition}[theorem]{Proposition}
\newtheorem{lemma}[theorem]{Lemma}
\newtheorem{definition}{Definition}
\newtheorem{corollary}{Corollary}

\newtheorem{conjecture}{Conjecture}

\newcommand{\HP}{\mathcal{HP}}
\newcommand{\bR}{\mathbb{R}}
\newcommand{\bC}{\mathbb{C}}
\newcommand{\ee}{\end{equation}}
\newcommand {\al}{\alpha}

\newcommand {\la}{\lambda}
\newcommand{\be}{\beta}
\newcommand{\mesh}{\mathop{\mathrm{mesh}}}

\newcommand {\de}{\delta}
\newcommand {\ga}{\gamma}

\newcommand{\A}{\mathcal A}

\begin{document}
\title[Elements of P\'olya-Schur theory in finite difference setting]{Elements of P\'olya-Schur  theory in finite difference setting}

\author[P.~Br\"and\'en]{Petter Br\"and\'en}
\address{Department of Mathematics, Royal Institute of Technology, SE-100 44 Stockholm,
Sweden}
\email{pbranden@kth.se}

\author[I. Krasikov]{Ilia Krasikov}

\address{   Department of Mathematical Sciences,
            Brunel University,
            Uxbridge UB8 3PH United Kingdom}
\email{mastiik@brunel.ac.uk}

\author[B. Shapiro]{Boris Shapiro}

\address{   Department of Mathematics,
            Stockholm University,
            S-10691, Stockholm, Sweden}
\email{shapiro@math.su.se}

\begin{abstract}
The P\'olya-Schur theory describes the class of hyperbolicity preservers, i.e., the linear operators on univariate polynomials preserving real-rootedness.
We  attempt  to develop an analog of P\'olya-Schur theory  in the setting of linear finite difference operators. We study  the class of linear finite difference operators preserving the set of  real-rooted polynomials whose  mesh (i.e., the minimal distance between the roots) is at least  one. In particular,   finite difference version of the classical Hermite-Poulain theorem and several results about discrete multiplier sequences are obtained. 
\end{abstract}

\maketitle

\section{Introduction}
The systematic study of linear operators acting on $\mathbb R[x]$ and sending real-rooted polynomials to real-rooted polynomials  was initiated in the 1870's by C.~Hermite and  later continued by E.~Laguerre. Its classical period culminated in 1914 with the publication of the outstanding paper \cite {PolyaSchur} where G.~P\'olya and I.~Schur completely characterized all such linear operators acting diagonally on the standard monomial basis $1, x, x^2, \ldots$ of $\bR[x]$. 
This article generated  a  substantial amount of related literature with  contributions by e.g. N.~Obreschkov, S.~Karlin, B.~Ya.~Levin, G.~Csordas, T.~Craven,  A.~Iserles, S.~P.~N\o rsett, E.~B.~Saff, and, recently by the first author together with 
 the late J.~Borcea. 

Although several variations of the original set-up have been considered (including complex zero decreasing sequences, real-rooted polynomials on finite intervals, stable polynomials etc.) it seems that its natural finite difference analog discussed below has so far escaped the attention of the specialists in the area.  An exception is \cite{Fisk}. 

Denote by $\HP\subset \mathbb {R}[x]$ the set of all real-rooted  (also referred to as {\em hyperbolic}) polynomials. A linear operator $T:\bR[x]\to\bR[x]$ is called a {\em real-rootedness preserver} or a 
{\em hyperbolicity preserver} if it preserves $\HP$. 
Given a real-rooted polynomial $p(x)\in \HP$ denote by $\mesh(p)$ its mesh; i.e., the minimal distance between its roots. If a real-rooted $p(x)$ has a multiple root, then by definition $\mesh(p):=0$. Polynomials of degree at most $1$ are defined to have  mesh equal to $+\infty$. Denote by $\HP_{\ge \alpha}\subset  \HP$ the set of all real-rooted polynomials whose mesh is at least $\alpha \geq 0$. 
Let $\HP^{+}_{\ge \alpha} \subset \HP_{\ge \alpha}$ be the subset of such polynomials with only non-negative zeros.

One of rather few known results about linear operators not decreasing the mesh is due originally to M.~Riesz and deserves to be better known, see e.g. \cite{Fisk,Sto}.

\begin{theorem}\label{th:Riesz}  For any hyperbolic polynomial $p$ and any real $\lambda$, 
$$\mesh(p-\lambda p') \ge \mesh(p).$$
\end{theorem} 

Recall that the well-known Hermite-Poulain theorem \cite[p.~4]{obr} claims that a finite order linear  differential operator  
$T=a_0+a_1{d}/{dx}+\cdots+a_k{d^k}/{dx^k}$ with constant coefficients is hyperbolicity preserving  if and only if its {\it symbol polynomial} $Q_T(t)=a_0+a_1t+\cdots+a_kt^k$ is hyperbolic. Thus  Theorem~\ref{th:Riesz} combined with the Hermite-Poulain theorem imply the following statement.  

\begin{corollary}\label{cor:1}
A hyperbolicity preserving differential operator with constant coefficients does not decrease the mesh of hyperbolic polynomials. 
\end{corollary}

\medskip
Our first goal  is to find an analog of  Corollary~\ref{cor:1}  in the finite difference context. 
We  consider  the action on $\bC[x]$ of   linear finite difference operators $T$ with polynomial coefficients; i.e.,  operators of the form:
\begin{equation}\label{eq:basic}
T(p)(x)=q_0(x)p(x)+q_1(x)p(x-1)+\cdots+q_k(x)p(x-k),
\end{equation}
where $q_0(x),\ldots,q_k(x)$ are  fixed complex- or real-valued polynomials.  If $q_k(x)\not \equiv 0$  we say that $T$ has order $k$. 
Although no  non-trivial  $T$ as in \eqref{eq:basic} preserves $\HP$, (see Lemma~\ref{lm:1} below) it can nevertheless   preserve $\HP_{\ge 1}$. The simplest example of such  an operator  is 
$$\Delta(p(x))=p(x)-p(x-1)$$
 which is a discrete analog of ${d}/{dx}$, see Fig.~1.  

\begin{figure}

\begin{center}
\includegraphics[scale=0.75]{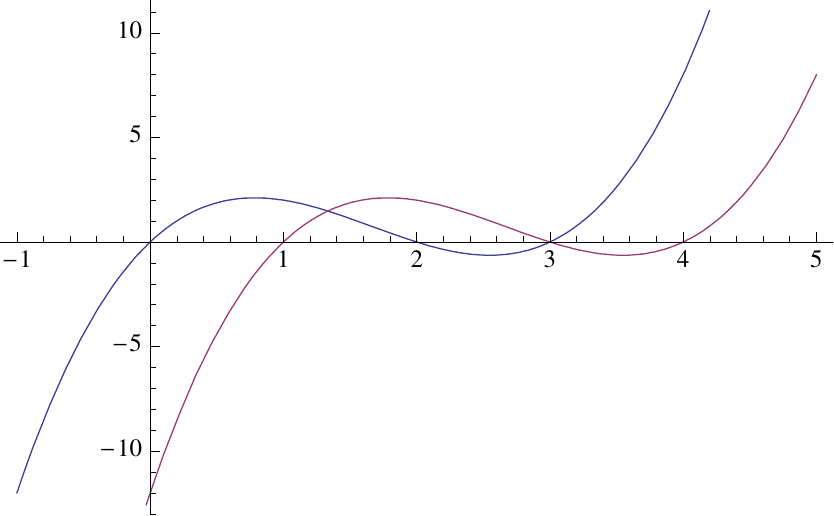}
\end{center}


\caption{Roots of $\Delta(p(x))=p(x)-p(x-1)$ are the $x$-coordinates of the intersection points between the graphs of $p(x)$ and $p(x-1)$.}
\label{fig1}
\end{figure}

\begin{definition}
 A linear finite difference operator \eqref{eq:basic} is called a  {\bf discrete hyperbolicity preserver} if it preserves $\HP_{\ge 1}$. 
 \end{definition}
Obviously,  the set of all discrete hyperbolicity preservers is a semigroup with respect to composition.  We start with a  finite difference analog of  Theorem~\ref{th:Riesz}. A similar result was proved by S.~Fisk in \cite[Lemma 8.27]{Fisk}. 

\begin{theorem}\label{th:FDRiesz} For  positive real numbers  $\al$ and  $\lambda$,  define 
an  operator $T$ by 
$$T(p)(x)= p(x)-\lambda p(x-\al).$$ Then for any  hyperbolic polynomial $p\in \HP_{\ge \alpha}$,  
$$\mesh(T(p)) \ge \mesh(p).$$ 
Moreover if $\lambda \ge 1$, then   $T$ preserves the set $\HP_{\ge \alpha}^{+}$. 
\end{theorem} 

This statement settles Conjecture 2.19 from a recent preprint \cite{CCh}. Our next result is a natural finite difference analog of the Hermite-Poulain theorem. 

\begin{theorem}Ê\label{th:HerPou} A linear finite difference operator $T$ with constant coefficients of the form
\begin{equation}\label{eq:constcoeff}
T(p(x))=a_0p(x)+a_1p(x-1)+\cdots+a_kp(x-k)
\end{equation}  
is a discrete hyperbolicity preserver if and only if all zeros of its symbol polynomial $Q(t)=a_0+a_1t+\cdots+a_kt^k$ are real and non-negative.   
\end{theorem}Ê 
As we mentioned above, a famous class of hyperbolicity preservers   is the class of  multiplier sequences introduced and studied by G.~P\'olya and I.~Schur in \cite{PolyaSchur}.  Let us recall this  notion   and  introduce its finite difference analog.

\begin{definition}  Given a sequence $\A=\{\al_i\}_{i=0}^\infty$ of real or complex numbers, we denote by $T_\A$ the linear operator  
$$T_\A(x^i)=\al_ix^i$$
acting diagonally with respect to the monomial basis of $\bC[x]$. 
We refer to $T_\A$ as the {\bf diagonal operator} corresponding to the sequence $\A$. 
\end{definition}

Notice that any diagonal operator $T$ as above can be also written as a formal linear differential operator of (in general) infinite order 
$$T=\sum_{i=0}^\infty a_i x^i \frac{d^i}{dx^i}.$$ 
The relation between the sequences $\A=\{\al_i\}_{i=0}^\infty$ and $A=\{a_i\}_{i=0}^\infty$ representing the same diagonal operator $T$ is of triangular form and given by:  
$$\al_i=a_0+ia_1+i(i-1)a_2+\dots+i!a_i,\; i=0,1,2,\dots $$

\begin{definition}
  We call a sequence $\A=\{\al_i\}_{i=0}^\infty$ of real numbers  a {\bf multiplier sequence of the 1st kind},  if its diagonal operator $T_\A$   preserves $\HP$; i.e., sends an arbitrary hyperbolic polynomial to a hyperbolic polynomial. The above sequence $\A$ is called a  {\bf multiplier sequence of the 2nd kind},  if the above $T_\A$ sends an arbitrary hyperbolic polynomial whose roots are all of the same sign to a hyperbolic polynomial. 
  \end{definition}Ê
The main results of  \cite{PolyaSchur} are explicit criteria  describing when a given sequence $\A=\{\al_i\}_{i=0}^\infty$ represents  a multiplier sequence of the 1st and the 2nd kind.  

Let us now describe a finite difference version of multiplier sequences.  Natural analogs of monomials in the  finite difference setting are the  Pochhammer polynomials  $\{(x)_i\}_{i=0}^\infty$  defined by
 \begin{equation}\label{eq:Poch}
 (x)_0=1,\quad (x)_i=x(x-1)\cdots(x-i+1),\; i\ge 1.
 \end{equation}

\begin{definition}
A finite difference operator $T$ as in \eqref{eq:basic} is called {\bf diagonal} if it acts diagonally with respect to the Pochhammer basis $\{(x)_i\}_{i=0}^\infty$.  
\end{definition}

Analogously to the above case of the usual diagonal operators we can associate to any sequence  $\A=\{\al_i\}_{i=0}^\infty$ of real numbers the corresponding 
diagonal finite difference operator $T_\A$ (in general, of infinite order) by assigning 
$$T_\A((x)_i)=\al_i(x)_i,\; i=0, 1, 2, \ldots$$ 
Observe that a finite difference  analog  $x\Delta$ of the Euler operator $x\frac{d}{dx}$ given by 
$$x\Delta=x(p(x)-p(x-1))$$
 acts diagonally in this basis, namely, $x\Delta((x)_i)=i(x)_i.$ Moreover any diagonal finite difference operator $T$ (of finite or infinite order) can be represented as a formal series 
 $$T=\sum_{i=0}^\infty a_i (x)_i \Delta^i.$$

\begin{definition}
We say that  a diagonal finite difference operator $T$ is  a  {\bf discrete multiplier sequence} if it preserves $\HP^+_{\ge 1}$.
\end{definition}

Our next result is as follows. 

\begin{theorem}\label{th:FDM} An operator $U$ given by $$U(p(x))=\al p(x)+\be x\Delta(p(x))=\al p(x)+\be x(p(x)-p(x-1)),$$ 
 is a discrete multiplier sequence if $\al$ and $\be$ are real numbers of the same sign.  
\end{theorem} 
\noindent 
{\bf Remark 1.}  Observe that, in general, the above operator $U$   is {\bf not} Êmesh-increasing. Therefore, Theorem~\ref{th:FDM}Ê is not a complete analog of Theorem~\ref{th:FDRiesz}.  A simple example of this phenomenon is 
$U(p(x))=p(x)+(3/{4}) x(p(x)-p(x-1))$; i.e., $\al=1,\quad \be={3}/{4}$. When $p(x)=(x-1)(x-4)(x-7)$, then $U(p(x))$ has three positive roots which are approximately equal to $(0.433167, 3.12467, 6.36524)$ and its mesh is smaller than $3$. 

\begin{proposition}\label{d-c}
If $\{\al_i\}_{i=0}^\infty$ is a discrete multiplier sequence, then it is a multiplier sequence in the classical sense. 
\end{proposition}
\noindent 
{\bf Remark 2.} 
Notice that the converse to Proposition~\ref{d-c} fails  since the ordinary multiplier sequence  $\{\rho^i\}_{i=0}^\infty$,  where $0 < \rho <1$, is not a a discrete multiplier sequence. 

\medskip
Denote by $\mathcal{L\!\!-\!\!P_+}$  the  positive subclass in the  Laguerre-P\'olya class; i.e., real entire functions which are the uniform limits, on compact subsets of the complex plane, of polynomials with only real positive zeros.  

\begin{theorem}\label{phi-k}
If  $\phi(x) \in \mathcal{L\!\!-\!\!P}_+$ then the sequence $\{\phi(i)\}_{i=0}^\infty$ is a discrete multiplier sequence. 
\end{theorem}

A sequence $\{\alpha_i\}_{i=0}^\infty$ is said to be \emph{trivial} if $\alpha_i \neq 0$ for at most two indices $i$. Trivial discrete multiplier sequences are simple to describe. 

\begin{proposition}\label{trivial}
A trivial sequence $\{\alpha_i\}_{i=0}^\infty$  is a discrete multiplier sequence if and only if there is an integer $m \geq 0$ such that $\alpha_m\alpha_{m+1}\geq 0$ and $\alpha_i = 0$ unless $i \in \{m,m+1\}$. 
\end{proposition}

We conjecture the following tantalizing characterization of non-trivial discrete multiplier sequences, which would be a discrete parallel to the classical result of P\'olya and Schur \cite{PolyaSchur}.  

\begin{conjecture}\label{nice}
Let  $\{\alpha_i\}_{i=0}^\infty$ be a non-trivial sequence such that $\alpha_i>0$ for some $i$. Then it is a discrete multiplier sequence if and only if it is a multiplier sequence such that $0 \leq \alpha_1 \leq \alpha_2 \leq \cdots$. 
\end{conjecture}

We almost prove one direction of Conjecture \ref{nice}, namely we prove  that any discrete multiplier sequence with infinitely many non-zero entries and at least one positive entry is weakly increasing, see Proposition \ref{alink}.  

%
%

\medskip
\noindent
{\it Acknowledments.} The authors are grateful to Professors O.~Katkova and A.~Vishnyakova of Kharkov National University for discussions of the topic. The third author is grateful to the Department of Mathematics, Brunel University for the hospitality in June 2009 when this project was initiated. 

\medskip

\section{Proving a Discrete Hermite-Poulain theorem}Ê
The following lemma  emphasizes the difference between ordinary and discrete hyperbolicity preservers. 
\begin{lemma}\label{lm:1}
A  finite difference operator $T$  of the form \eqref{eq:basic} is hyperbolicity preserving in the classical sense if and only if $q_i(x) \not \equiv 0$ for at most one $i$, and $q_i(x)$ is hyperbolic for such an $i$.
\end{lemma}

\begin{proof} If $T$ satisfies the conditions of the lemma, then $T$ is trivially a hyperbolicity preserver. 

Consider the bivariate symbol
$$
G(x,y)= T(e^{-xy}) = \sum_{j=0}^kq_j(x)e^{-(x-j)y}= e^{-xy}\sum_{j=0}^kq_j(x)e^{jy}.
$$
If $T$ is a hyperbolicity preserver, then  by \cite[Theorem 5]{julius1}, $G(x,y)$ or $G(x,-y)$ is the limit (uniform on compact subsets of $\bC$) of bivariate polynomials that are non-vanishing whenever $\text{Im }x >0$ Êand $\text{Im }y >0$. It follows that for each $x_0 \in \bR$ the function 
$$
\sum_{j=0}^kq_j(x_0)e^{\pm (x_0-j)y}
$$
is in the Laguerre--P\'olya class. However this is the case only if $q_j(x_0) \neq 0$ for at most one $j$, from which it follows that $q_j(x) \not \equiv 0$ for at most one $j$. Since $T(f)= q_j(x)f(x-j)$ this forces $q_j(x)$ to be hyperbolic. 
%
%
\end{proof}


Before we present a proof of Theorem~\ref{th:FDRiesz} 
we need  to recall some notation and  well known results about hyperbolic polynomials. 
Let $\ga_1 \leq \ga_2 \leq \cdots \leq \ga_n$ and $\de_1 \leq \de_2 \leq \cdots \leq \de_m$ be the zeros of two hyperbolic polynomials $p$ and $q$. These zeros \emph{interlace} if either $\ga_1 \leq \de_1 \leq \ga_2 \leq \de_2 \leq \cdots$ or $\de_1 \leq \ga_1 \leq \de_2 \leq \ga_2 \leq \cdots$. 
A pair of  hyperbolic polynomials $(p,q)$ are in \emph{proper position}, written $p \ll q$, if their zeros interlace and 
$p(x)q'(x) - p'(x)q(x) \geq 0$ for all $x \in \bR$.  Note that if the zeros of two hyperbolic polynomials $p$ and $q$ interlace,  then either $p \ll q$ or $q \ll p$. By convention we set $0 \ll p$ and $p \ll 0$ for any hyperbolic polynomial $p$. The next lemma follows from a simple count of sign changes, see \cite[Theorem 6.3.8]{RS}, \cite[Lemma 2.4]{W} and   
\cite[Lemma 2.6]{London}. 

\begin{lemma}\label{cone} \mbox{ } \\
\rm{(a)} Let $p$ be a hyperbolic polynomial. Then the sets 
$$
\{ q \in \bR[x] : q \ll p \} \quad \mbox{ and } \quad \{ q \in \bR[x] : p \ll q \} 
$$
are convex cones. 

\noindent
\rm{(b)} If $p \ll q$,  then $p \ll q + \alpha p$ and $p+\alpha q \ll q$ for all $\alpha \in \bR$.  
\end{lemma}

\begin{proof}[Proof of Theorem \ref{th:FDRiesz}]
Let $T(p)(x)= p(x)-\lambda p(x-\alpha)$ where $\alpha, \lambda \geq 0$. We want to prove that $T : \HP_{\geq \beta} \rightarrow \HP_{\geq \beta}$ for all $\beta \geq \alpha$. First we prove it for $\beta = \alpha$. Note that $p \in \HP_{\geq \alpha}$ if and only if $p(x) \ll p(x-\alpha)$. Lemma \ref{cone} (b) implies $T(p)(x) \ll p(x)$ and $T(p)(x) \ll p(x-\alpha)$, which is easily seen to imply $T(p) \in \HP_{\geq \alpha}$. Next we prove that if $p, q \in \HP_{\geq \alpha}$ satisfy $p \ll q$, then $T(p) \ll T(q)$. This will prove Theorem  \ref{th:FDRiesz}, since if 
$p \in   \HP_{\geq \beta} \subseteq  \HP_{\geq \alpha}$, then $p(x) \ll p(x-\beta)$ and thus $T(p)(x) \ll T(p)(x-\beta)$ which is equivalent to 
$T(p) \in   \HP_{\geq \beta} $.  By a continuity argument invoking Hurwitz' theorem on the continuity of zeros \cite[Theorem 1.3.8]{RS} we may assume that $p$ and $q$ have the same degree. To prove that $T$ preserves proper position we claim that it is enough to prove that 
\begin{equation}\label{ab}
T( (x-a)r) \ll T( (x-b)r) 
\end{equation}
whenever $a \leq b$ and $r \in  \HP_{\geq \alpha}$. Indeed, let $a$ be the smallest zero of $p$, $b$ be the greatest zero of $q$ and $r=p/(x-a)$. Then we may construct a sequence of hyperbolic polynomials:
$$
(x-a)r=p=p_0 \ll p_1 \ll \cdots \ll q \ll \cdots \ll p_k = (x-b)r, 
$$
where for each $1\leq i \leq k-1$ there is a factorization  $p_{i-1}(x)= (x-a_i)r_i(x)$ and 
$p_{i}(x)= (x-b_i)r_i(x)$, where $a_i \leq b_i$ and $r_i(x) \in  \HP_{\geq \alpha}$. By hypothesis 
$$
T((x-a)r)=T(p)=T(p_0)  \ll \cdots \ll T(q) \ll \cdots \ll T(p_k) = T((x-b)r), 
$$
and $T((x-a)r) \ll T((x-b)r)$ which implies $T(p) \ll T(q)$ as claimed.

It remains to prove \eqref{ab}. Since $T((x-b)p)= T((x-a)p)-(b-a)T(p)$ it is, by Lemma \ref{cone} (a) and invariance under translation, enough to prove that 
$T(p) \ll T(xp)$ for all  $p \in  \HP_{\geq \alpha}$. Now 
$$
T(xp)= (x-\alpha)T(p)+ \alpha p.
$$
Since $T(p) \ll p$ and $T(p) \ll (x-\alpha)T(p)$, Lemma \ref{cone} b implies $T(p) \ll  (x-\alpha)T(p)+ \alpha p$ as desired. 

Finally suppose $p \in  \HP_{\geq \alpha}^+$ and $\lambda \geq 1$. Write 
$p(x)= A\prod_{i=1}^n(x-\theta_i)$ where $\theta_i \geq 0$ for all $i$. Then 
for $y\geq 0$ 
$$
\frac {p(-y)}{p(-y-1)} = \prod_{i=1}^n \frac {y+\theta_i} {y+\theta_i+1} <1 \leq \lambda.  
$$
Hence $T(p)(-y) \neq 0$, which proves that $T$ preserves $\HP_{\geq \alpha}^+$. 
\end{proof}


\begin{proof}[Proof of  Theorem~\ref{th:HerPou}]  Theorem~\ref{th:FDRiesz} implies  that if the symbol polynomial $Q_T(t)=a_0+a_1t+\cdots+a_kt^k$ has only  real and non-negative zeros,  then the finite difference operator $T(p(x))=a_0p(x)+a_1p(x-1)+\cdots+a_kp(x-k)$ is a discrete hyperbolicity preserver.   We need to prove the necessity of the latter condition. Consider the action of $T$ on the Pochhammer polynomials $(x)_i$. 
Assuming that $i\ge k$, we get 
$$T((x)_i)=(x-k)\cdots(x-i+1)R_i(x),$$
where $R_i(x)$ is a hyperbolic polynomial of degree $k$. Observe that 
\begin{equation}\label{limR}
\lim_{i\to \infty}\frac{R_i(ix)}{i^k}=x^kQ\left(\frac{x-1}{x}\right),
\end{equation}
where $Q_T(t)$ is the above symbol polynomial. Hence if $T$ is a discrete hyperbolicity preserver, then $Q_T(t)$ is  hyperbolic. We need to show that its zeros are non-negative. Suppose that $Q_T(y)=0$ for $y<0$. The assumption $y=(x-1)/{x}$ implies  $0<x<1$.  By \eqref{limR} and Hurwitz' theorem on the continuity of zeros  it follows that there are real numbers $0<a<b<1$ and an integer $i_0$ such that 
$R_i(ix)$ has a zero in the interval $(a,b)$ whenever $i>i_0$. Hence $R_i(x)$ has a zero in $(ia,ib)$ for all $i >i_0$. If we choose $i>i_0$ large enough so that $(ia,ib)\subset (k,i)$, we see that the mesh of  $T((x)_i)=(x-k)\cdots(x-i+1)R_i(x)$ 
is strictly smaller than $1$ which contradicts  our assumption. Hence all zeros of $Q_T(t)$ are non-negative.  
\end{proof}

\section {Proving results on discrete multiplier sequences}

 To prove Theorem~\ref{th:FDM}  it suffices to consider the operator: 
$$W_{\la}(p)=p(x)+\lambda x\Delta(p(x))= p(x)+ \lambda x \left( p(x)-  p(x-1) \right).$$

\begin{proposition}
\label{claim2} For each $\lambda \ge 0$,
$$W_{\la}(p) : \HP^{+}_{\ge 1} \rightarrow \HP^+_{\ge 1}, $$
i.e. $W_{\la}$ is a discrete multiplier sequence. 
 \end{proposition}

\begin{proof} 
Let $p \in \HP^{+}_{\ge 1}$. As in the proof of Theorem \ref{th:FDRiesz} we observe that 
$$
p(x) - p(x-1) \ll p(x) \mbox{ and } p(x) - p(x-1) \ll p(x-1). 
$$
Since the degree of $p(x)-p(x-1)$ is one less than that of $p(x)$ and since all the zeros of $p(x)-p(x-1)$ are non-negative (they interlace those of $p$) we have 
$$
\lambda x(p(x) - p(x-1)) \ll p(x) \mbox{ and } \lambda x(p(x) - p(x-1)) \ll p(x-1). 
$$
Since $p(x) \ll p(x)$ and $p(x) \ll p(x-1)$, Lemma \ref{cone} (a) implies 
$$
W_\lambda(p)(x) \ll p(x) \mbox{ and } W_\lambda(p)(x) \ll p(x-1),
$$
which in turn implies $W_\lambda(p) \in \HP_{\ge 1}$. Since 
$$
\HP^{+}_{\ge 1} \ni x(p(x) - p(x-1)) \ll p(x) \in \HP^{+}_{\ge 1},
$$
these polynomials have the same sign for negative real numbers which implies $W_\lambda(p) \in \HP^+_{\ge 1}$. 
\end{proof}

The next result is due to F.~Brenti \cite{Bre}. We provide a proof here for completeness.  

\begin{lemma}\label{brenti}
Let $T : \bR[x] \rightarrow \bR[x]$ be defined by 
$$
T(x^i) = (x)_i.
$$
If all the zeros of the polynomial $p(x)$ are real and non-negative, then 
$T(p) \in \HP_{\geq 1}^+$. 
\end{lemma}

\begin{proof}
We prove Lemma~\ref{brenti}  by induction on $n$, the degree of $p$. The cases $n=0$ and $1$ are trivial so assume $p(x)$ is polynomial of degree $n+1 \geq 2$ and write 
$$
p(x) = (x-\alpha)q(x)=(x-\alpha)\sum_{i=0}^n \ga_i x^i,
$$
where $\alpha \geq 0$. 
By induction we know that $Q(x) = T(q) \in \HP_{\geq 1}^+$.  An elementary manipulation shows
$$
T(p)= xQ(x-1) -\alpha Q(x).
$$
Since $Q(x)  \in \HP_{\geq 1}^+$:
$$
xQ(x-1) \ll -Q(x), -\alpha Q(x) \ll -Q(x), -xQ(x-1) \ll -Q(x-1), -\alpha Q(x)  \ll -Q(x-1), 
$$
so by Lemma \ref{cone} (a):
$$
T(p) \ll -Q(x) \quad \mbox{ and } \quad T(p) \ll -Q(x-1),
$$
which proves $T(p)  \in \HP_{\geq 1}^+$. 
\end{proof}

\begin{proof}[Proof of Proposition~\ref{d-c}]
Suppose that all zeros of a test polynomial $p(x)=\ga_0+\ga_1x+ \cdots +\ga_nx^n$ are real and non-negative. By Lemma \ref{brenti} 
$$
\sum_{i=0}^n \ga_i \rho^i (x)_i \in \HP_{\geq 1}^+
$$
for all $\rho >0$. But then 
$$
\sum_{i=0}^n \ga_i \rho^i \al_i (x)_i \in \HP_{\geq 1}^+ \mbox{ and thus } \sum_{i=0}^n \ga_i \rho^i \al_i (x/\rho)_i \in \HP_{\geq \rho}^+
$$
for all $\rho >0$. Letting $\rho \to 0$ we see that 
$$
\sum_{i=0}^n \ga_i \al_i x^i \in \HP_{\geq 0}^+,
$$
and hence $\{\al_i\}_{i=0}^\infty$ is an ordinary multiplier sequence. 
\end{proof}

\begin{proof}[Proof of Theorem~\ref{phi-k}] 
 Proposition \ref{claim2} claims that the sequence $\{1+\lambda i\}_{i=0}^\infty$ is a discrete multiplier sequence for each $\lambda \geq 0$. Since the set of all discrete multiplier sequences is a semi-group under composition all hyperbolic polynomials with negative zeros  give rise to discrete multiplier sequences  via $p \mapsto \{p(i)\}_{i=0}^\infty$. The set $\mathcal{L\!\!-\!\!P}_+$ is the closure of such polynomials, from which the theorem follows.
\end{proof}

\begin{lemma}\label{altn}
Suppose $p(x) = \sum_{i=0}^n a_i (x)_i \in \HP_{\geq 1}^+$ with $a_n >0$.  
Then $(-1)^{n-i}a_i \geq 0$ for all $0\leq i \leq n$. 
\end{lemma}
\begin{proof}
Since $p(x)$ has $n$ non-negative zeros and $p(x) >0$ for $x>0$ large enough, we have $(-1)^{n}p(0)=(-1)^{n}a_0 \geq 0$. As in the proof of Proposition \ref{claim2}, we see that $\nabla(p) \ll p$ and 
$\nabla(p) \in \HP_{\geq 1}^+$. Here 
$\nabla p(x)= p(x+1)-p(x)$ is the forward difference operator.  Now 
$$
\nabla(p) = \sum_{i=0}^{n-1} (i+1)a_{i+1} (x)_i,
$$
and the lemma follows by iterating the argument for $i=0$. 
\end{proof}

An immediate consequence of Lemma \ref{altn} is: 

\begin{corollary}\label{signs}
All non-zero entries of a discrete multiplier sequence have the same sign.
\end{corollary}

Next we give the proof of the characterization of trivial discrete multiplier sequences. 

\begin{proof}[Proof of Proposition \ref{trivial}]
Suppose $\{\al_i\}_{i=0}^\infty$ is a trivial discrete multiplier sequence. Then, by Corollary \ref{signs}, we may assume that all entries are nonnegative. The only if direction now follows from the well known fact that all nonnegative and trivial multiplier sequences are of the desired form.

Assume that the sequence satisfies the conditions in the statement of the proposition with $\al_m\al_{m+1} \geq 0$. Let 
$$
p(x) = \sum_{i=0}^n a_i(x)_i, 
$$
 and $T$ be the diagonal finite difference operator associated to $\{\al_i\}_{i=0}^\infty$. Then 
 $$
 T(p)(x)= \al_m a_m (x)_m + \al_{m+1} a_{m+1} (x)_{m+1}= 
 -a(x)_m +b(x)_{m+1}, 
 $$
 where $ab\geq 0$ by Lemma \ref{altn}. If $b =0$ we are done, so assume $b>0$. Then 
$$
T(p)(x)= b(x)_m(x-m-b/a) \in \HP_{\geq 1}^+,
$$
as desired. 
\end{proof}
\begin{proposition}\label{alink}
Let $\{\al_i\}_{i=0}^\infty$ be a discrete multiplier sequence. If $\al_{m+2} >0$ for some $m \geq 0$, then $\al_m \leq \al_{m+1}$. 
\end{proposition}

\begin{proof}
Let $a\geq 0$ and consider 
\begin{align*}
&T\big( (x)_m (x-m-a)(x-1-m-a) \big) \\
&=  (x)_m \big( \al_{m+2}(x-m)(x-m-1)-2a\al_{m+1}(x-m)+\al_{m}a(a+1) \big)
\end{align*}
A polynomial $Ax(x-1) -2Bx + C$ where $A,B,C \geq 0$ is in $\HP_{\geq 1}^+$ if and only if $AC \leq B^2+AB$, which yields 
$$
0 \leq a(\al_{m+1}^2-\al_m\al_{m+2})+\al_{m+2}(\al_{m+1}-\al_m), \quad \mbox{ for all } a \geq 0.
$$
In particular $\al_{m+1} \geq \al_m$. 
\end{proof}

%

Let us present more examples of  discrete multiplier sequences. 
\medskip

\noindent
$\bullet$ For any non-negative $i$ the operator $(x)_i \Delta^i$ is a discrete multiplier sequence, i.e. $\{(n)_i\}_{n=0}^\infty$ is a discrete multiplier. It follows from the fact that if $p\in \HP_{\ge 1}^+$ then 
$
\Delta p(x) \ll p(x-1) 
$ 
and therefore all zeros of $\Delta p$ are in $[1,\infty)$. 

\noindent
$\bullet$ For any non-negative $i$ and any polynomial $q$ with all roots in $(-\infty, i]$,  the sequence  $\{p(i)\}_{i=0}^\infty$ where 
$$
p(x)= (x)_i q(x)
$$
is a discrete multiplier sequence.

\section{Final remarks} 

Ê The Hermite-Poulain theorem has a  version in finite degrees (which we could not find explicitly stated in the literature): 

\begin{proposition}\label{pr:finitedegree} A differential operator $T=a_0+a_1{d}/{dx}+\cdots+a_k{d^k}/{dx^k},\; a_k\neq 0$ with constant coefficients preserves the set of   hyperbolic polynomial of degree at most $m$  if and only if the polynomial $T(x^m)$ is hyperbolic.  
\end{proposition}
Proposition \label{pr:finitedegree}  follows immediately from the algebraic characterization of hyperbolicity preservers, Theorem 2 of \cite{julius1}. 

The role of monomials $x^m$ in the finite difference setting is often played by the Pochhammer polynomials \eqref{eq:Poch}. 
In particular, Proposition~\ref{pr:finitedegree} might have the following conjectural analog in the finite difference setting.

\begin{conjecture}\label{pr:finitedegreeDif} A difference operator $T(p(x))=a_0p(x)+a_1p(x-1)+\cdots+a_kp(x-k)$ with constant coefficients preserves the set of  hyperbolic polynomial of degree at most $m$ whose mesh is at least one  if and only if the polynomial $T((x)_m)$ is hyperbolic and has mesh at least one.  
\end{conjecture}
There is an alternative formulation of Conjecture \ref{pr:finitedegreeDif} which is maybe more attractive. Let 
$\nabla p(x)= p(x+1)-p(x)$ be the forward difference operator, and consider the following product on the space of polynomials of degree at most $d$:
$$
(p \bullet q)(x) = \sum_{k=0}^d (\nabla^kp)(0) \cdot(\nabla^{d-k}q)(x).
$$
Conjecture \ref{pr:finitedegreeDif} is equivalent to
\begin{conjecture}\label{pr:finitedegreeDif2} If $p$ and $q$ are hyperbolic polynomials of degree at most $d$ and of mesh $\geq 1$, then so is $p \bullet q$. 
\end{conjecture}



%

\end{document}